\documentclass{arxiv}
\usepackage{amsmath, amssymb, amsthm, graphicx}
\usepackage{comment}
\newtheorem{theorem}{Theorem}[section]
\newtheorem{lemma}[theorem]{Lemma}

\theoremstyle{definition}

\newtheorem{criterion}[]{Criterion}
\newtheoremstyle{named}{}{}{\itshape}{}{\bfseries}{.}{.5em}{\thmnote{#3's }#1} \theoremstyle{named} 

\theoremstyle{remark}

\numberwithin{equation}{section}

\numberwithin{equation}{section}

\title{Existence of a class of rotopulsators}
\author{Pieter Tibboel}
\address{\em Department of Mathematics \\ Y6524 (Yellow Zone) 6/F Academic 1\\
City University of Hong Kong, Hong Kong\\
Tat Chee Avenue\\
Kowloon Tong\\
Hong Kong}
\email{\em ptibboel@cityu.edu.hk}
\begin{document}
\maketitle
\begin{abstract}
  We prove the existence of a class of rotopulsators for the $n$-body problem in spaces of constant curvature of dimension $k\geq 2$.
\end{abstract}

\section{Introduction}
  By $n$-body problems, we mean problems where we want to find the dynamics of $n$ point particles. If the space in which such a problem is defined is a space of zero curvature, then we call any solution to such a problem for which the point particles describe the vertices of a polytope that retains its shape over time (but not necessarily its size) a homographic orbit. \\
  A rotopulsator, also known as a rotopulsating orbit, is a type of solution to an $n$-body problem for spaces of constant curvature $\kappa\neq 0$ that extends the definition of homographic orbits to spaces of constant curvature (see \cite{DK}). \\
  Homographic orbits (and therefore rotopulsators) can be used to determine the geometry of the universe locally (see for example \cite{D2}, \cite{DK}). \\
  In this paper, we will prove the existence of a subclass of rotopulsators that form a natural generalisation of orbits found in \cite{D2} and \cite{D3}. \\
  While this paper mainly builds on results obtained in \cite{D2}, \cite{D3} and \cite{T},  research on $n$-body problems for spaces of constant curvature goes back to Bolyai \cite{BB} and Lobachevsky \cite{Lo}, who independently proposed a curved 2-body problem in hyperbolic space $\mathbb{H}^{3}$ in the 1830s. In later years, $n$-body problems for spaces of constant curvature have been studied by mathematicians such as Dirichlet, Schering \cite{S1}, \cite{S2}, Killing \cite{K1}, \cite{K2}, \cite{K3} and Liebmann \cite{L1}, \cite{L2}, \cite{L3}. More recent results were obtained by Kozlov, Harin \cite{KH}, but the study of $n$-body problems in spaces of constant curvature for the case that $n\geq 2$ started with \cite{DPS1}, \cite{DPS2}, \cite{DPS3}  by Diacu, P\'erez-Chavela, Santoprete. Further results for the $n\geq 2$ case were then obtained by Cari\~nena, Ra\~nada, Santander \cite{CRS}, Diacu \cite{D1}, \cite{D2}, \cite{D3}, Diacu, Kordlou \cite{DK}, Diacu, P\'erez-Chavela \cite{DP}. For a more detailed historical overview, please see \cite{D2}, \cite{D3}, \cite{D4}, \cite{DK}, or \cite{DPS1}.\\
  In this paper, we will prove the following two theorems:
  \begin{theorem}\label{Main Theorem1}
    For any rotopulsating solution of (\ref{EquationsOfMotion}) formed by vectors $\{\textbf{q}_{i}\}_{i=1}^{n}$ as defined in (\ref{Expression homographic orbits}), the vectors $\{\textbf{Q}_{i}\}_{i=1}^{n}$ have to form a regular polygon if $\rho$ is non-constant.
  \end{theorem}
  \begin{theorem}\label{Main Theorem2}
    Rotopulsating orbits formed by vectors $\{\textbf{q}_{i}\}_{i=1}^{n}$ as defined in (\ref{Expression homographic orbits}) exist if the vectors $\{\textbf{Q}_{i}\}_{i=1}^{n}$ form a regular polygon.
  \end{theorem}
To prove these theorems, we will use a method strongly inspired by \cite{D2}, \cite{D3} and \cite{T}. Specifically, we will first deduce a necessary and sufficient criterion for the existence of rotopulsators. This will be done in section~\ref{A criterion for the existence of rotopulsators}. We will then prove Theorem~\ref{Main Theorem1} and Theorem~\ref{Main Theorem2} in section~\ref{Proof of Main Theorem 1} and section~\ref{Proof of Main Theorem 2} respectively.
\section{A criterion for the existence of rotopulsators}\label{A criterion for the existence of rotopulsators}
In this section, we will formulate a necessary and sufficient criterion for the existence of rotopulsating orbits of the type described in (\ref{Expression homographic orbits}). \\
Consider the $n$-body problem in spaces of constant curvature $\kappa\neq 0$. \\
As has been shown in \cite{D4}, we may assume that $\kappa$ equals either $-1$, or $1$. \\
We will denote the masses of its $n$ point particles to be $m_{1}$, $m_{2}$,..., $m_{n}>0$ and their positions by the $k$-dimensional vectors
\begin{align*}\textbf{q}_{i}^{T}=(q_{i1},q_{i2},...,q_{ik})\in\textbf{M}_{\kappa}^{k-1},\textrm{ }
i=\overline{1,n}\end{align*}
where
\begin{align*}\textbf{M}_{\kappa}^{k-1}=\{(x_{1},x_{2},...,x_{k})\in\mathbb{R}^{k}|\textrm{
}\kappa(x_{1}^{2}+x_{2}^{2}+...+x_{k-1}^{2}+\sigma x_{k}^{2})=1\},\textrm{ }k\in\mathbb{N}\end{align*}
and
\begin{align*}
    \sigma=\begin{cases}
      \hspace{0.25cm}1 &\textrm{ for }\kappa> 0\\
      -1 &\textrm{ for }\kappa<0
    \end{cases}.
  \end{align*}
Furthermore, consider for $m$-dimensional vectors $\textbf{a}=(a_{1},a_{2},...,a_{m})$, \\ $\textbf{b}=(b_{1},b_{2},...,b_{m})$ the inner product
  \begin{align}\label{odot inner product}
    \textbf{a}\odot_{m}\textbf{b}=a_{1}b_{1}+a_{2}b_{2}+...+a_{m-1}b_{m-1}+\sigma
    a_{m}b_{m}.
  \end{align}
Then, following \cite{D1}, \cite{D2}, \cite{D3}, \cite{DPS1}, \cite{DPS2}, \cite{DPS3} and the assumption that $\kappa=\pm 1$ from \cite{D4}, we define the equations of motion for the curved $n$-body problem as the dynamical system described by
\begin{align}\label{EquationsOfMotion}
  \ddot{\textbf{q}}_{i}=\sum\limits_{j=1, j\neq i}^{n}\frac{m_{j}[\textbf{q}_{j}-(\sigma\textbf{q}_{i}\odot_{k}\textbf{q}_{j})\textbf{q}_{i}]}{[\sigma-(\textbf{q}_{i}\odot_{k}\textbf{q}_{j})^{2}]^{\frac{3}{2}}}-(\sigma\dot{\textbf{q}}_{i}\odot_{k}\dot{\textbf{q}}_{i})\textbf{q}_{i},\textrm{ }i=\overline{1,n}.
\end{align}
Let
\begin{align*}
  T(t)=\begin{pmatrix}
    \cos(\theta(t)) & -\sin(\theta(t)) \\
    \sin(\theta(t)) & \cos(\theta(t))
  \end{pmatrix}
\end{align*}
be a $2\times 2$ rotation matrix, where $\theta(t)$ is some real valued, twice continuously differentiable, scalar function, for which $\theta(0)=0$. \\
Furthermore, let $\rho(t)$ be a nonnegative, twice continuously differentiable, scalar function. \\
We will consider rotopulsating orbit solutions of (\ref{EquationsOfMotion}) of the form
\begin{align}\label{Expression homographic orbits}
  \textbf{q}_{i}(t)=\begin{pmatrix}\rho(t)T(t)\textbf{Q}_{i}\\ Z(t)\end{pmatrix}
\end{align}
where $\textbf{Q}_{i}\in\mathbb{R}^{2}$ is a constant vector and $Z(t)\in\mathbb{R}^{k-2}$ is a twice differentiable, vector valued function. \\
Finally, before formulating our criterion, we need to introduce some notation and a lemma:\\
Let $m\in\mathbb{N}$. Let $\langle\cdot,\cdot\rangle_{m}$ be the Euclidean inner product on $\mathbb{R}^{m}$ and let $\|\cdot\|_{m}$ be the Euclidean norm on $\mathbb{R}^{m}$.
Let $i$, $j\in\{1,...,n\}$. By construction $\|Q_{i}\|_{2}=\|Q_{j}\|_{2}$ for all $i$, $j\in\{1,...,n\}$ and we will assume that  $\|\textbf{Q}_{i}\|_{2}=1$. Let $\beta_{i}$ be the angle between $\textbf{Q}_{i}$ and the first coordinate axis.
The lemma we will need to prove our criterion is:
\begin{lemma}\label{The Lemma}
 The functions $\rho$ and $\theta$, are related through the following formula:  $\rho^{2}(t)\dot{\theta}(t)=\rho^{2}(0)\dot{\theta}(0)$.
\end{lemma}
\begin{proof}
  In \cite{D3}, using the wedge product, Diacu proved that
  \begin{align*}
    \sum\limits_{i=1}^{n}m_{i}\dot{\textbf{q}}_{i}\wedge\textbf{q}_{i}=\textbf{c}
  \end{align*}
  where $\textbf{c}$ is a constant bivector. \\
  If $\{\textbf{e}_{i}\}_{i=1}^{k}$ are the standard base vectors in $\mathbb{R}^{k}$, then we can write $\textbf{c}$ as
  \begin{align}\label{cwedge1}
    \textbf{c}=\sum\limits_{i=1}^{k}\sum\limits_{j=1}^{k}c_{ij}\textbf{e}_{i}\wedge\textbf{e}_{j}.
  \end{align}
  where $\{c_{ij}\}_{i=1,j=1}^{k}$ are constants. As $\textbf{e}_{i}\wedge \textbf{e}_{j}=-\textbf{e}_{j}\wedge\textbf{e}_{i}$ and $\textbf{e}_{i}\wedge\textbf{e}_{i}=0$ (see \cite{D3}), for $i$, $j\in\{1,...,n\}$, we can rewrite (\ref{cwedge1}) as
  \begin{align}
    \textbf{c}=\sum\limits_{i=1}^{k}\sum\limits_{j=i+1}^{k}C_{ij}\textbf{e}_{i}\wedge\textbf{e}_{j}
  \end{align}
  where $C_{ij}=c_{ij}-c_{ji}$. \\
  Calculating $C_{12}$,  will give us our result: \\
  Note that  \begin{align}\label{properties of T}T^{T}=T^{-1}\textrm{ and }\dot{T}=\dot{\theta}\begin{pmatrix}
    0 & -1\\ 1 & 0
  \end{pmatrix}T\end{align}
  and
  \begin{align}
    C_{12}&=\sum\limits_{i=1}^{n}m_{i}\left(q_{i1}\dot{q}_{i2}-q_{i2}\dot{q}_{i1}\right)\nonumber\\
    &=\sum\limits_{i=1}^{n}m_{i}\left(q_{i1},q_{i2}\right)\begin{pmatrix}
      0 & 1\\ -1 & 0
    \end{pmatrix}\begin{pmatrix}\dot{q}_{i1}\\
    \dot{q}_{i2}\end{pmatrix}.\label{c2l-12l formula 1}
  \end{align}
  Using (\ref{Expression homographic orbits}) with (\ref{c2l-12l formula 1}) gives
  \begin{align}
    C_{12}&=\sum\limits_{i=1}^{n}m_{i}\rho^{2}\left(Q_{i1},Q_{i2}\right)T^{T}\begin{pmatrix}
      0 & 1\\ -1 & 0
    \end{pmatrix}\dot{T}\begin{pmatrix}Q_{i1}\\
    Q_{i2}\end{pmatrix}\nonumber\\
    &+\sum\limits_{i=1}^{n}m_{i}\rho\dot{\rho}\left(Q_{i1},Q_{i2}\right)T^{T}\begin{pmatrix}
      0 & 1\\ -1 & 0
    \end{pmatrix}T\begin{pmatrix}Q_{i1}\\
    Q_{i2}\end{pmatrix}.\label{c2l-12l formula 2}
  \end{align}
  Note that
  \begin{align*}
    \rho\dot{\rho}\left(Q_{i1},Q_{i2}\right)T^{T}\begin{pmatrix}
      0 & 1\\ -1 & 0
    \end{pmatrix}T\begin{pmatrix}Q_{i1}\\
    Q_{i2}\end{pmatrix}\\ =\frac{\dot{\rho}}{\rho}(q_{i1},q_{i2})\begin{pmatrix}
      0 & 1\\ -1 & 0
    \end{pmatrix}\begin{pmatrix}
      q_{i1}\\ q_{i2}
    \end{pmatrix}=0.
  \end{align*}
  So, using (\ref{properties of T}) repeatedly, we get that
  \begin{align*}
    C_{12}&=\sum\limits_{i=1}^{n}m_{i}\rho^{2}\dot{\theta}\left(Q_{i\textrm{ }1},Q_{i2}\right)T^{T}\begin{pmatrix}
      0 & 1\\ -1 & 0
    \end{pmatrix}\begin{pmatrix}0 & -1\\ 1 & 0\end{pmatrix}T\begin{pmatrix}Q_{i1}\\
    Q_{i1}\end{pmatrix}+0\\
    &=\sum\limits_{i=1}^{n}m_{i}\rho^{2}\dot{\theta}\left(Q_{i1},Q_{i2}\right)T^{T}T\begin{pmatrix}Q_{i1}\\
    Q_{i2}\end{pmatrix}\\
    &=\sum\limits_{i=1}^{n}m_{i}\rho^{2}\dot{\theta}\left(Q_{i1},Q_{i2}\right)\begin{pmatrix}Q_{i1}\\
    Q_{i2}\end{pmatrix}
  \end{align*}
  which means that
  \begin{align}\label{c2l-12l final formula}
    C_{12}=\rho^{2}\dot{\theta}\sum\limits_{i=1}^{n}m_{i}\left(Q_{i1}^{2}+Q_{i2}^{2}\right). \end{align}
As, by construction
\begin{align*}
  \sum\limits_{i=1}^{n}m_{i}\left(Q_{i1}^{2}+Q_{i2}^{2}\right)>0,
\end{align*}
we may divide both sides of (\ref{c2l-12l final formula}) by
\begin{align*}
  \sum\limits_{i=1}^{n}m_{i}\left(Q_{i1}^{2}+Q_{i2}^{2}\right),
\end{align*}
which gives that
\begin{align*}
  \rho^{2}\dot{\theta}=\frac{C_{12}}{\sum\limits_{i=1}^{n}m_{i}\left(Q_{i1}^{2}+Q_{i2}^{2}\right)},
\end{align*}
which is constant, so $\rho^{2}\dot{\theta}=\rho^{2}(0)\dot{\theta}(0)$.
\end{proof}
  We now have the following necessary and sufficient criterion for the existence of a rotopulsating orbit, as described in (\ref{Expression homographic orbits}):
\begin{criterion}\label{criterion}
  Let
  \begin{align}\label{bi}
    b_{i}=\sum\limits_{j=1, j\neq i}^{n}\frac{m_{j}(1-\cos{(\beta_{i}-\beta_{j}}))^{-\frac{1}{2}}}{(2-\sigma\rho^{2}(1-\cos{(\beta_{i}-\beta_{j})}))^{\frac{3}{2}}}.
  \end{align}
  Then necessary and sufficient conditions for the existence of a rotopulsating orbit of non-constant size are that
  $b_{1}=b_{2}=...=b_{n}$ and
  \begin{align}\label{c}
    0=\sum\limits_{j=1, j\neq i}^{n}\frac{m_{j}\sin{(\beta_{i}-\beta_{j})}}{(1-\cos{(\beta_{i}-\beta_{j})})^{\frac{3}2}(2-\sigma\rho^{2}(1-\cos{(\beta_{i}-\beta_{j})}))^{\frac{3}{2}}}
  \end{align}
  for all $i\in\{1,...,n\}$.
\end{criterion}
\begin{proof}
  Note that
  \begin{align}\label{Expression homographic orbits 2}
    \dot{T}=\dot{\theta}T\begin{pmatrix}
      0 & -1\\ 1 & 0
    \end{pmatrix}
  \end{align}
  and consequently
  \begin{align}\label{Expression homographic orbits 3}
    \ddot{T}=\ddot{\theta}T\begin{pmatrix}
      0 & -1\\ 1 & 0
    \end{pmatrix}-\dot{\theta}^{2}T.
  \end{align}
  Inserting (\ref{Expression homographic orbits}) into (\ref{EquationsOfMotion}) and using (\ref{Expression homographic orbits 2}) and (\ref{Expression homographic orbits 3}) gives for the first and second lines of (\ref{EquationsOfMotion}) that
  \begin{align}\label{ToTheCriterion1}
    T\left(\ddot{\rho}I_{2}+2\dot{\rho}\dot{\theta}\begin{pmatrix}
      0 & -1\\ 1 & 0
    \end{pmatrix}+\rho\left(\ddot{\theta}\begin{pmatrix}
      0 & -1\\ 1 & 0
    \end{pmatrix}-\dot{\theta}^2 I_{2}\right)\right)\textbf{Q}_{i}\nonumber\\
    =\rho T\left(\sum\limits_{j=1, j\neq i}^{n}\frac{m_{j}[\textbf{Q}_{j}-(\sigma \textbf{q}_{i}\odot_{k}\textbf{q}_{j})\textbf{Q}_{i}]}{[\sigma-(\textbf{q}_{i}\odot_{k}\textbf{q}_{j})^{2}]^{\frac{3}{2}}}-(\sigma\dot{\textbf{q}}_{i}\odot_{k}\dot{\textbf{q}}_{i})\textbf{Q}_{i}\right)
  \end{align}
  where $I_{2}$ is the $2\times 2$ identity matrix. \\
  For the last $k-2$ lines, we get
  \begin{align}\label{Expression Zodot Z 1}
    \ddot{Z}=\left(\sum\limits_{j=1, j\neq i}^{n}\frac{m_{j}[1-(\sigma \textbf{q}_{i}\odot_{k}\textbf{q}_{j})]}{[\sigma-(\textbf{q}_{i}\odot_{k}\textbf{q}_{j})^{2}]^{\frac{3}{2}}}-(\sigma\dot{\textbf{q}}_{i}\odot_{k}\dot{\textbf{q}}_{i})\right)Z.
  \end{align}
  Note that
  \begin{align}\label{Expression Zodot Z 2}
    \textbf{q}_{i}\odot_{k}\textbf{q}_{j}=\rho^{2}\langle \textbf{Q}_{i},\textbf{Q}_{j}\rangle_{2}+Z\odot_{k-2} Z.
  \end{align}
  As we have that $\langle \textbf{Q}_{i},\textbf{Q}_{i}\rangle_{2}=1$ and as by (\ref{Expression Zodot Z 2}), \begin{align*}\sigma^{-1}=\textbf{q}_{i}\odot_{k}\textbf{q}_{i}=\rho^{2}\langle\textbf{Q}_{i},\textbf{Q}_{i}\rangle_{2}+Z\odot_{k-2} Z, \end{align*} we may rewrite (\ref{Expression Zodot Z 2}) as
  \begin{align*}
    \textbf{q}_{i}\odot_{k}\textbf{q}_{j}=\sigma^{-1}+\rho^{2}\langle \textbf{Q}_{i},\textbf{Q}_{j}\rangle_{2}-\rho^{2},
  \end{align*}
  which can, in turn, be written as
  \begin{align}\label{Final expression qiqj}
    \textbf{q}_{i}\odot_{k}\textbf{q}_{j}=\sigma^{-1}+\rho^{2}(\cos{(\beta_{i}-\beta_{j})}-1).
  \end{align}
  Furthermore,
  \begin{align}\label{qdotqdot 1}
    \dot{\textbf{q}}_{i}\odot_{k}\dot{\textbf{q}}_{i}=\langle\dot{\rho}T\textbf{Q}_{i}+\rho\dot{T}\textbf{Q}_{i},\dot{\rho}T\textbf{Q}_{i}+\rho\dot{T}\textbf{Q}_{i}\rangle_{2}+\dot{Z}\odot_{k-2}\dot{Z}.
  \end{align}
  As $T$ is a rotation in $\mathbb{R}^{2}$, it is a unitary map, meaning that for $v$, $w\in\mathbb{R}^{2}$, $\langle Tv,Tw\rangle_{2}=\langle v,w\rangle_{2}$, meaning that (\ref{qdotqdot 1}) can be written as
  \begin{align}\label{qdotqdot 2}
    \dot{\textbf{q}}_{i}\odot_{k}\dot{\textbf{q}}_{i}=\langle\dot{\rho}\textbf{Q}_{i}+\rho T^{-1}\dot{T}\textbf{Q}_{i},\dot{\rho}\textbf{Q}_{i}+\rho T^{-1}\dot{T}\textbf{Q}_{i}\rangle_{2}+\dot{Z}\odot_{k-2}\dot{Z}.
  \end{align}
  Using (\ref{Expression homographic orbits 2}) with (\ref{qdotqdot 2}) gives
  \begin{align}\label{qdotqdot 3}
    \dot{\textbf{q}}_{i}\odot_{k}\dot{\textbf{q}}_{i}&=\dot{\rho}^{2}+2\rho\dot{\rho}\dot{\theta}\left\langle \textbf{Q}_{i},\begin{pmatrix}
       0 & -1\\ 1& 0
     \end{pmatrix}\textbf{Q}_{i}\right\rangle_{2}+\rho^{2}\dot{\theta}^{2}\|\textbf{Q}_{i}\|^{2}+\dot{Z}\odot_{k-2}\dot{Z}\nonumber \\
     &=\dot{\rho}^{2}+0+\rho^{2}\dot{\theta}^{2}+\dot{Z}\odot_{k-2}\dot{Z}.
  \end{align}
  Inserting (\ref{qdotqdot 3}) and (\ref{Final expression qiqj}) into (\ref{ToTheCriterion1}) and multiplying both sides by $T^{-1}$ provides us with
  \begin{align}\label{ToTheCriterion3}
    &\left((\ddot{\rho}-\rho\dot{\theta}^2)I_{2}+(2\dot{\rho}\dot{\theta}+\rho\ddot{\theta})\begin{pmatrix}
      0 & -1\\ 1 & 0
    \end{pmatrix}\right)\textbf{Q}_{i}\nonumber\\
    &=\rho\sum\limits_{j=1, j\neq i}^{n}\frac{m_{j}[\textbf{Q}_{j}-(1-\sigma\rho^{2}(1-\cos{(\beta_{i}-\beta_{j})}))\textbf{Q}_{i}]}{[\rho^{2}(1-\cos{(\beta_{i}-\beta_{j})})(2-\sigma\rho^{2}(1-\cos{(\beta_{i}-\beta_{j})}))]^{\frac{3}{2}}}\\ &-(\sigma\rho\dot{\rho}^{2}+\sigma\rho^{3}\dot{\theta}^{2}+\sigma\rho\dot{Z}\odot_{k-2}\dot{Z})\textbf{Q}_{i}. \nonumber
  \end{align}
  Taking the Euclidean inner product with $\textbf{Q}_{i}$ on both sides of (\ref{ToTheCriterion3}) and using that $\|Q_{i}\|_{2}=\|Q_{j}\|_{2}=1$, provides us with
  \begin{align}\label{ToTheCriterion3b}
    &\ddot{\rho}-\rho\dot{\theta}^2+\sigma\rho\dot{\rho}^{2}+\sigma\rho^{3}\dot{\theta}^{2}+\sigma\rho\dot{Z}\odot_{k-2}\dot{Z}\nonumber\\
    &=\left(\sigma-\frac{1}{\rho^{2}}\right)\sum\limits_{j=1, j\neq i}^{n}\frac{m_{j}[(1-\cos{(\beta_{i}-\beta_{j})})^{-\frac{1}{2}}]}{[(2-\sigma\rho^{2}(1-\cos{(\beta_{i}-\beta_{j})}))]^{\frac{3}{2}}}.
  \end{align}
  Taking the Euclidean inner product of (\ref{ToTheCriterion3}) with $\begin{pmatrix}0 & 1\\ -1& 0\end{pmatrix}\textbf{Q}_{i}$ and using that $\|Q_{i}\|_{2}=\|Q_{j}\|_{2}=1$ gives that
  \begin{align}\label{sin}
    2\dot{\rho}\dot{\theta}+\rho\ddot{\theta}=\sum\limits_{j=1, j\neq i}^{n}\frac{m_{j}\sin{(\beta_{i}-\beta_{j})}}{[(1-\cos{(\beta_{i}-\beta_{j})})(2-\sigma\rho^{2}(1-\cos{(\beta_{i}-\beta_{j})}))]^{\frac{3}{2}}}.
  \end{align}
  Let
  \begin{align*}
    &b_{i}:=\sum\limits_{j=1, j\neq i}^{n}\frac{m_{j}[(1-\cos{(\beta_{i}-\beta_{j})})^{-\frac{1}{2}}]}{[(2-\sigma\rho^{2}(1-\cos{\alpha_{ij}}))]^{\frac{3}{2}}}\textrm{ and }\\
    &c_{i}:=\sum\limits_{j=1, j\neq i}^{n}\frac{-m_{j}\sin{(\beta_{i}-\beta_{j})}}{[(1-\cos{(\beta_{i}-\beta_{j})})(2-\sigma\rho^{2}(1-\cos{(\beta_{i}-\beta_{j})}))]^{\frac{3}{2}}}.
  \end{align*}
  Inserting (\ref{qdotqdot 3}) and (\ref{Final expression qiqj}) into (\ref{Expression Zodot Z 1}), combined with (\ref{ToTheCriterion3b}) and (\ref{sin}) gives the following system of differential equations:
  \begin{align}\label{system of differential equations}
    \begin{cases}
      \ddot{\rho}=\rho\dot{\theta}^2-\sigma\rho\dot{\rho}^{2}-\sigma\rho^{3}\dot{\theta}^{2}-\sigma\rho\dot{Z}\odot_{k-2}\dot{Z}+\left(\sigma-\frac{1}{\rho^{2}}\right)b_{i}\\
      \ddot{\theta}=\frac{c_{i}}{\rho}-2\frac{\dot{\rho}}{\rho}\dot{\theta}\\
      \ddot{Z}=\left(b_{i}-\sigma\dot{\rho}^{2}-\sigma\rho^{2}\dot{\theta}^{2}-\sigma\dot{Z}\odot_{k-2}\dot{Z})\right)Z
    \end{cases}
  \end{align}
  For (\ref{system of differential equations}) to make sense, we need that
  \begin{align}\label{conditions}
    b_{1}=...=b_{n}\textrm{ and }c_{1}=...=c_{n}
  \end{align}
  which shows the necessity of (\ref{conditions}). \\
  Furthermore, that (\ref{system of differential equations}) has a global solution holds by the same argument as the argument used in the proof of Criterion 1 in \cite{D2} to prove global existence of a solution of (15) and (17). By the uniqueness of solutions to ordinary differential equations given suitable initial conditions, the solution to (\ref{system of differential equations}) must be a rotopulsating orbit, as every step from (\ref{ToTheCriterion1}) and (\ref{Expression Zodot Z 1}) to (\ref{system of differential equations}) is invertible. \\
  Thus (\ref{conditions}) is both necessary and sufficient. Finally, as by Lemma~\ref{The Lemma}\\
  $\rho^{2}\dot{\theta}=\rho^{2}(0)\dot{\theta}(0)$, we have that $\frac{d}{dt}(\rho^{2}\dot{\theta})=0$, which means that the left hand side of (\ref{sin}) equals zero, which means that $c_{i}=0$. This completes the proof.
\end{proof}
\section{Proof of Theorem~\ref{Main Theorem1}}\label{Proof of Main Theorem 1}
In Criterion~\ref{criterion}, let $r:=\rho$, $\alpha_{i}:=\beta_{i}$,  $\delta_{i}:=b_{i}$ and $\gamma_{i}:=c_{i}$.
Then the conditions of Criterion~\ref{criterion} become exactly the conditions of Criterion~1 in \cite{D2} with the added bonus that $\gamma_{i}=0$. The proof of Theorem~1.1 in \cite{T} is therefore a proof for Theorem~\ref{Main Theorem1} as well.
\section{Proof of Theorem~\ref{Main Theorem2}}\label{Proof of Main Theorem 2}
Let again $r:=\rho$, $\alpha_{i}:=\beta_{i}$,  $\delta_{i}:=b_{i}$ and $\gamma_{i}:=c_{i}$ in Criterion~\ref{criterion}. Then the conditions of Criterion~\ref{criterion} become exactly the conditions of Criterion~1 in \cite{D2} with the added bonus that $\gamma_{i}=0$. Theorem~\ref{Main Theorem2} now follows directly from the proofs of Theorem~1 and Theorem~2 in \cite{D2}.
\section{Acknowledgements}
The author is indebted to Florin Diacu and Dan Dai for all their advice.

\end{document}